%% file: main.tex
\documentclass[10pt,twoside]{article}
\input{preamble.tex}

\setArticleTitleWidth{0.9\textwidth}
\title{The graded Betti numbers are not tropical invariants}
\runningtitle{The graded Betti numbers are not tropical invariants}
\runningauthor{N. Laikin}

\paperauthors{%
  \authorblock{Noah Laikin}{BSc in Mathematics, Durham University\\ Durham, U.K.}%
}

\begin{document}
\maketitle

\begin{abstract}
\noindent We show that the graded Betti numbers are not invariants of the tropicalisation of a homogeneous ideal, or the tropicalisation of a projective algebraic scheme over a valued field. We do this by producing a family of
saturated ideals that all realise the same tropical ideal with one member having a distinct Betti table from the rest. 
\end{abstract}

\section{Introduction}

A motivating question of tropical geometry is \emph{what information held by algebro-geometric objects over a valued field can be extracted from tropicalisation?} The classical setting is the tropicalisation of varieties
over valued fields, as treated in \citep{MaclaganSturmfels2015}.

To recover additional geometric information about these varieties, Giansiracusa-Giansiracusa subsequently defined tropicalisations of algebraic schemes over valued fields in \citep{GiansiracusaGiansiracusa2016} whose rational
points naturally correspond to the tropicalisation of their underlying varieties. Moreover, using properties of tropical linear spaces, Maclagan-Rincón showed in \citep{MaclaganRincon2020} that the tropicalisation of a
projective algebraic scheme over a valued field recovers the tropicalisation of its defining saturated ideal.

Following their introduction by Maclagan-Rinc\'{o}n in \citep{MaclaganRincon2018}, much theory has since been developed for ideals of tropical algebras which are tropical linear spaces, called \emph{tropical ideals}
\citep{MaclaganRincon2018}, \citep{MaclaganRincon2022}, \citep{FinkEtAl2025}, though many questions still remain open. It is known that tropical schemes and their associated tropical ideals contain additional geometric
information tropical varieties discard. In particular, a homogeneous tropical ideal \(J\) has a well-defined Hilbert function which, when \(J\) is realisable, agrees with the Hilbert function of its realisations 
\citep[p.~651]{MaclaganRincon2018}.

It is hence natural to ask if the graded Betti numbers are also numerical invariants of the tropicalisation of an ideal, i.e.\ if the Betti tables of the realisations of a homogeneous tropical ideal all agree. In this paper,
we prove \cref{thm:main}, which when the base field has at least two valuation zero elements gives distinct realisations of a homogeneous tropical ideal with distinct Betti tables, demonstrating that this is false in general.

For notational convenience, fix a field \(k\) and a valuation \(\nu : k\rightarrow \mathbf{R}\cup\{\infty\}\). Let tropicalisation be with respect to \(\nu\), let \(S\coloneq k[x,y,z]\) with its standard grading, and let \(\mathbf{T}\) denote the tropical semiring \((\mathbf{R}\cup\{\infty\},\min,+)\).
We prove the following result: 

\begin{theorem}\label{thm:main}
There is a family of pairwise distinct saturated realisations \((J_{\lambda})_{\nu(\lambda)=0}\) of a homogeneous tropical ideal of \(\mathbf{T}[x,y,z]\) such that the graded Betti
numbers of \(S/J_{1}\) differ from those of \(S/J_{\lambda}\) for all other \(\lambda\).
\end{theorem}

Moreover, since these ideals are saturated, and as established by \citep{MaclaganRincon2020}, homogeneous saturated tropical ideals are in bijective correspondence with their tropical schemes, the graded Betti
numbers are also not generally numerical invariants of the tropicalisation of a scheme. 

We proceed by constructing a family of pairwise distinct homogeneous saturated ideals \((J_{\lambda})_{\lambda\in k}\) of \(S\) defining length-four projective schemes with three non-colinear points which for \(\lambda\neq 0\) lie in
the standard affine chart \(D_{+}(z)\simeq \mathbf{A}^{2}\). We then show that \(S/J_{1}\) has distinct graded Betti numbers from \(S/J_{\lambda}\) for \(\lambda\neq 1\), and that the \(J_{\lambda}\) for which
\(\nu(\lambda)=0\) all realise the same tropical ideal, proving \cref{thm:main}. 

\begin{acknowledgements}
I thank Jeffrey Giansiracusa for suggesting this question, and for taking the time to review my drafts of this paper.  
\end{acknowledgements}

\section{The counterexample}

As mentioned, the counterexample in this paper concerns a family of saturated ideals defining a length four scheme with three non-colinear points in \(\mathbf{P}^{2}\). Up to a homogeneous change of coordinates, any such ideal
has a minimal primary decomposition of the form 
\[
I_{\alpha} \coloneq (x,y)\cap (x,z) \cap (y^{2},\,\alpha y + z), \quad \alpha\in k.
\]
Indeed the intersection of saturated ideals is saturated \citep[Ch.~II, \S2.4, (1)]{Bourbaki1989CommutativeAlgebra}, and moreover a projective scheme \(X\) on \([1:0:0]\) with defining saturated ideal \(\mathcal{J}(X)\) 
is of length two exactly when its stalk \(k[y,z]/\mathcal{J}(X)^{x=1}\) is two-dimensional over \(k\). So we must have \(\mathcal{J}(X)=(y^{2},\,\alpha y+z)\) for some \(\alpha\in k\). Applying Dedekind's identity gives
\[
I_{\alpha}=(\alpha xy+xz,yz,xy^2), \quad \alpha\in k,
\]
and so the \(S/I_{\alpha}\) have two distinct Betti configurations, one when \(\alpha=0\), and another when \(\alpha\neq 0\). In particular, a degree three minimal syzygy is introduced when 
\(\alpha=0\) as then the first two generators share \(z\) as a common factor. 

We can visualise this degeneration by making the homogeneous change of coordinate \(z\mapsto x+(1-\alpha)y+z\) on \(I_{\alpha}\), which away from \(\alpha=1\) lands all points in the affine plane. In particular,
writing \(\lambda=1-\alpha\), we obtain the ideals
\[
J_{\lambda}\coloneq (x^{2}+xy+xz,\,xy+\lambda y^{2}+yz,\,xy^{2}), \quad \lambda\in k,
\]
and applying the change of coordinates to the primary decomposition, we can see that the \(J_{\lambda}\) have the minimal primary decompositions
\[
J_{\lambda}=(x,y)\cap (x,\,x+\lambda y +z) \cap (y^{2},\,x+ y + z), \quad \lambda\in k.
\]
Thus the reduced points of \(\mathrm{Proj}\.(S/J_{\lambda})\) are \([0:0:1]\), \([0:1:-\lambda]\), and its length two point is \([-1:0:1]\). Hence when \(\lambda\neq 0\), the scheme is isomorphic to 
its dehomogenisation \(\mathrm{Spec}\.(k[x,y]/J_{\lambda}^{z=1})\) under the canonical morphism induced by sending \(z\mapsto 1\).

\cref{fig:generic-special-fibre} provides a visualisation of this family of schemes. The lines which are solid rather than dashed at a point cut out its isolated primary component.
We can see that the degree three syzygy at \(\lambda=1\) corresponds to when the direction of the line \(x+\lambda y+1=0\) coincides with the tangent direction \((1,-1)\) of the
isolated primary component supported at \((-1,0)\) indicated by the arrow in each diagram.

\begin{figure}[h]
\caption{Representation of the fibres over the points in \(D(\lambda -1)\) and the special fibre at \(\lambda=1\) of the parametrised family 
\(\mathrm{Spec}\.(k[x,y,\lambda^{\pm 1}]/J_{\lambda}^{z=1})\rightarrow \mathrm{Spec}\.(k[\lambda^{\pm 1}])\)}
\centering
\vspace{0.2em}
\begin{tikzpicture}[
  scale=1.35, 
  scheme line/.style={thick, draw=gray!80},
]


\node[font=\normalfont, inner sep=0pt] at (0,-2.5) {\makebox[0pt][c]{General fibre at a point of \(D(\lambda-1)\)}};

\draw[scheme line] (-2,0) -- (2,0);
\node[anchor = east, inner sep = 0pt, font = \small] at (2,0.2) {$y=0$};

\draw[scheme line] (0,2) -- (0,-2);
\node[inner sep = 0pt, font = \small] at (0,2.2) {$x=0$};

\draw[scheme line] (-2,1) -- (1,-2);
\node[anchor = east, inner sep = 0pt, font = \small] at (2,-1.3) {$x+y+1=0$};

\draw[scheme line] (-0.5,0.5) -- (1,2);
\draw[scheme line, dash pattern = on 3.6pt off 3.6pt] (-0.5,0.5) -- (-2,-1);
\node[anchor = west, inner sep = 0pt, font = \small] at (-2,-1.2) {$x+\lambda y+1=0$};

\node[circle, fill, inner sep = 1.5pt] at (0,0) {};
\node[circle, fill, inner sep = 1.5pt] at (0,1) {};
\node[anchor = west, inner sep = 0pt, font = \small] at (0.2,1) {$(0,-1/\lambda)$};
\node[circle, fill, inner sep = 1.5pt] at (-1,0) {};
\draw[arrows = {-Latex[width=0pt 10, length=10pt]}] (-1,0) -- (-0.6,-0.4);


\node[font=\normalfont, inner sep=0pt] at (5.5,-2.5) {\makebox[0pt][c]{Special fibre at \(\lambda=1\)}};

\draw[scheme line] (3.5,0) -- (7.5,0);
\node[anchor = east, inner sep = 0pt, font = \small] at (7.5,0.2) {$y=0$};

\draw[scheme line] (5.5,2) -- (5.5,-2);
\node[inner sep = 0pt, font = \small] at (5.5,2.2) {$x=0$};

\draw[scheme line] (3.5,1) -- (6.5,-2);
\node[anchor = east, inner sep = 0pt, font = \small] at (7.5,-1.3) {$x+y+1=0$};

\node[circle, fill, inner sep = 1.5pt] at (5.5,0) {};
\node[circle, fill, inner sep = 1.5pt] at (5.5,-1) {};
\node[circle, fill, inner sep = 1.5pt] at (4.5,0) {};
\draw[arrows = {-Latex[width=0pt 10, length=10pt]}] (4.5,0) -- (4.9,-0.4);
\end{tikzpicture}

\label{fig:generic-special-fibre}
\vspace{0.5em}
\end{figure}
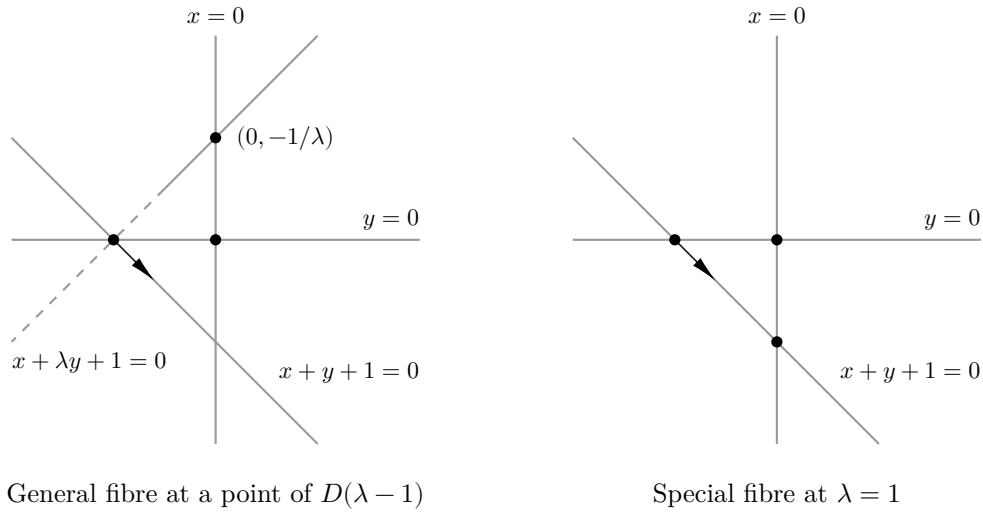

The \(J_{\lambda}\) such that \(\nu(\lambda)=0\) are pairwise distinct saturated realisations of a tropical ideal satisfying \cref{thm:main}. The pairwise distinctness follows from the \(\mathrm{Proj}\.(S/J_{\lambda})\) being pairwise distinct for all \(\lambda\in k\).
The saturatedness follows since saturatedness is stable under a homogeneous change of coordinates.

To prove they satisfy the remaining properties of the theorem, we start by computing the Betti table and Hilbert function of \(S/J_{\lambda}\) for all \(\lambda\in k\), verifying that the Betti table is indeed distinct when \(\lambda=1\) and
that the Hilbert functions are all the same. We then use this fact about the Hilbert functions to provide a simple proof that the tropicalisations of the \((J_{\lambda})_{\nu(\lambda)=0}\) are identical, proving the result. 

For notational convenience, we place \(\beta_{i,j}\) in column \(i\), row \(j-i\) of each Betti table.

\begin{lemma}[Betti tables of the \(J_{\lambda}\)]\label{lemma:betti}
Given our indexing convention, when \(\lambda\neq 1\), the quotient module \(S/J_{\lambda}\) has the Betti table 
\[
\begin{array}{c|ccc}
 & 0 & 1 & 2 \\
\hline
0 & 1 & \text{--} & \text{--} \\
1 & \text{--} & 2 & \text{--} \\
2 & \text{--} & \text{--} & 1
\end{array}
\]
Meanwhile, \(S/J_{1}\) has the Betti table
\[
\begin{array}{c|ccc}
 & 0 & 1 & 2 \\
\hline
0 & 1 & \text{--} & \text{--} \\
1 & \text{--} & 2 & 1 \\
2 & \text{--} & 1 & 1
\end{array}\vspace{0.3em}
\]
In particular, for all \(\lambda\in k\), the Hilbert function of \(S/J_{\lambda}\) equals
\[
H_{\lambda}(d) = \begin{dcases}
\.1 & d=0 \\ 
\.3 & d=1 \\ 
\.4 & d\ge 2
\end{dcases}
\]
and so is independent of the value of \(\lambda\).
\end{lemma}

\begin{proof}
When \(\lambda\neq 1\), we can write 
\[
(1-\lambda) xy^{2}= y\cdot (x^{2}+xy+xz)-x\cdot (xy+\lambda y^{2}+yz),
\]
and so we can write 
\[
J_{\lambda} = (x^{2}+xy+xz,\, xy+\lambda y^{2}+yz).\vspace{0.5em}
\]
This generating set is minimal since the two generators aren't linearly dependent over \(k\). Moreover they don't share a common factor, so the non-zero graded Betti numbers are
\[
\beta_{0,0}=1, \quad \beta_{1,2} = 2, \quad \beta_{2,4}=1,
\]
which gives the desired Betti table. 

We now consider the case where \(\lambda=1\). We can see that \(xy^{2}\) isn't in the \(S\)-linear span of the other two specified generators of \(J_{1}\), and so our
presentation is a minimal generating set. Hence the non-zero graded Betti numbers where \(i\le 1\) are
\[
\beta_{0,0}=1, \quad \beta_{1,2}=2, \quad \beta_{1,3}=1.
\]
Moreover, we have
\[
y\cdot (x^{2}+xy+xz)-x\cdot (xy+y^{2}+yz)=0
\]
and 
\[
y^{2}\cdot (x^{2}+xy+xz)-(x+y+z)\cdot xy^{2}=0,\vspace{0.5em}
\]
and so \((y,-x,0)\) and \((y^{2},0,-(x+y+z))\) are contained in the first syzygy module. They are \(S\)-linearly independent, and so have no syzygies, and hence provided they generate the entire
first syzygy module, the non-zero graded Betti numbers for \(i\ge 2\) are 
\[
\beta_{2,3}=1, \quad \beta_{2,4}=1,
\]
giving the result. Hence to prove that the Betti table of \(S/J_{1}\) is as claimed, it suffices to show that they do in fact generate the entire syzygy module, which we now show: 

Suppose that \(p,q,r\in S\) are such that 
\[
p\cdot (x^{2}+xy+xz)+q\cdot (xy+y^{2}+yz)+r\cdot xy^{2}=0.
\]
Rearranging, we obtain
\[
(x+y+z)\cdot (px+qy)+r\cdot xy^{2}=0,\vspace{0.2em}
\]
which implies \(r=-d\cdot (x+y+z)\) for some \(d\in S\), and we have \vspace{0.2em}
\[
(p-dy^{2})\cdot x+qy=0.\vspace{0.2em}
\]
Hence there exists an \(s\in S\) such that \(p=sy+dy^{2}\), \(q=-sx\), and so\vspace{0.3em}
\[
(p,q,r) = s\cdot (y,-x,0) + d\cdot (y^{2},0,-(x+y+z)),
\]
whence the result. 

The values of the Hilbert function for \(d=0,1,2\) follow from what we've already shown. Moreover, using the standard formula for the Hilbert function in terms of the
graded Betti numbers \citep[Cor.~1.10]{Eisenbud2005Syzygies}, we see that for all \(d>2\):\vspace{0.5em}
\[
\begin{aligned}
H_{\lambda}(d)&=\binom{d+2}{2}-2\cdot \binom{d}{2}+\binom{d-2}{2} \\[1em]
&= \frac{(d+1)(d+2)}{2}-d(d-1)+\frac{(d-2)(d-3)}{2} \\[1em]
               &= \frac{2d^{2}-2d+8}{2}-d^{2}+d = 4,
\end{aligned}
\vspace{0.2em}
\]
the result.
\end{proof}

We now use that they all share the same Hilbert function to prove that the \((J_{\lambda})_{\nu(\lambda)=0}\) all realise the same tropical ideal:

\begin{proposition}
The \((J_{\lambda})_{\nu(\lambda)=0}\) all realise the same tropical ideal, i.e.\ the tropical ideal \(\mathrm{trop}\.(J_{\lambda})\) doesn't depend on the choice of \(\lambda\).
\end{proposition}

\begin{proof}
By \cref{lemma:betti}, the Hilbert functions of the \(J_{\lambda}\) coincide, so the Hilbert functions of their tropicalisations also coincide. Hence by \citep[Thm.~3.11]{MaclaganRincon2018}, it suffices to show that
\[
\mathrm{trop}\.(J_{1})\subseteq \mathrm{trop}\.(J_{\lambda}) \quad \text{for all \(\lambda\) such that \(\nu(\lambda)=0\),}
\]
that is, for each \(f\in J_{1}\), to find a polynomial in \(J_{\lambda}\) which has the same tropicalisation. Since \(xy^{2}\in J_{\lambda}\), we can reduce to showing there is a
polynomial in \(J_{\lambda}\) whose coefficients of monomials not divisible by \(xy^{2}\) have the same valuations as those of \(f\). Writing 
\[
f\coloneq p\cdot (x^{2}+xy+xz) + q\cdot (xy+ y^{2}+yz) + r\cdot xy^{2}, \quad p,q,r\in S,
\]
we show that the polynomial
\[
f_{\lambda}\coloneq p\cdot (x^{2}+xy+xz) + q(x,\lambda y,z)\cdot (xy+\lambda y^{2}+yz) + r\cdot xy^{2}\in \.J_{\lambda}
\]
is satisfactory. Note that they share the same coefficient valuations on monomials not divisible by \(y^{2}\) since we have that
\(
q-q(x,\lambda y ,z) 
\)
is divisible by \(y\) and so 
\[
f-f_{\lambda} = y\cdot (x+\lambda y+z)\cdot (q-q(x,\lambda y,z))+y^{2}\cdot (1-\lambda)\cdot q
\]
is divisible by \(y^{2}\). Hence it remains to show that they share the same \pagebreak coefficient valuations on monomials not divisible by \(x\). Reducing mod \(x\) implies we can reduce to showing
that
\[
q(0,y,z)\cdot(y^{2}+yz), \quad q(0,\lambda y,z)\cdot (\lambda y^{2}+yz)
\]
have the same tropicalisation. Denote these polynomials by \(s\) and \(t\) respectively. Writing\vspace{0.3em}
\[
q(0,y,z)\coloneq\sum_{i,j\ge 0} c_{ij}y^{i}z^{j},\vspace{0.2em}
\]
and \(c_{ij}=0\) otherwise, the coefficient of \(y^{m}z^{n}\) in \(s\) is
\[
c_{m-2,n}+c_{m-1,n-1},
\]
whilst the coefficient of \(y^{m}z^{n}\) in \(t\) is
\[
\lambda^{m-1}\cdot (c_{m-2,n}+c_{m-1,n-1}).
\]
Since \(\nu(\lambda)=0\), these coefficients have the same valuation, which implies the result.
\end{proof} 

\bibliographystyle{plainnat}
\bibliography{references}
\end{document}

%% file: preamble.tex
\usepackage{iftex}
\ifPDFTeX
  \PackageError{journal-preamble}{Compile with LuaLaTeX or XeLaTeX}{This preamble uses fontspec and unicode-math.}
\fi

\usepackage{fontspec}

\usepackage[
  paperwidth=220mm,
  paperheight=293.5mm,
  twoside,
  inner=34mm,
  outer=51mm,
  top=42.5mm,
  textheight=193mm,
  headheight=14.7pt,
  headsep=15.75pt,
  footskip=30pt
]{geometry}

\usepackage{microtype}
\usepackage{tikz}
\usetikzlibrary{arrows.meta,calc}
\usepackage[labelfont=bf, justification=centering]{caption}
\usepackage{amsmath,mathtools}
\usepackage{unicode-math}

\usepackage{amsthm}
\usepackage{tikz-3dplot}
\usepackage{enumitem}
\usepackage{changepage}
\usepackage{titlesec}
\usepackage{fancyhdr}
\usepackage[numbers,sort&compress]{natbib}
\usepackage{xcolor}
\usepackage[colorlinks=true,linktoc=page,linkcolor=blue,citecolor=blue,urlcolor=blue]{hyperref}
\usepackage{aliascnt}
\usepackage[capitalise, noabbrev]{cleveref}

\makeatletter
\renewcommand\normalsize{%
  \@setfontsize\normalsize{10pt}{14pt}%
  \abovedisplayskip 10pt plus 2pt minus 1pt%
  \abovedisplayshortskip 6pt plus 2pt minus 1pt%
  \belowdisplayshortskip \abovedisplayshortskip
  \belowdisplayskip \abovedisplayskip
}
\renewcommand\small{\@setfontsize\small{9pt}{13.1pt}}
\renewcommand\footnotesize{\@setfontsize\footnotesize{8pt}{10pt}}
\renewcommand\scriptsize{\@setfontsize\scriptsize{7pt}{9pt}}
\renewcommand\tiny{\@setfontsize\tiny{5pt}{6pt}}
\renewcommand\large{\@setfontsize\large{12pt}{14pt}}
\renewcommand\Large{\@setfontsize\Large{17.28pt}{20pt}}
\makeatother
\normalsize

\makeatletter
\newcommand{\paper@runningtitle}{}
\newcommand{\paper@runningauthor}{}
\newcommand{\runningtitle}[1]{\gdef\paper@runningtitle{#1}}
\newcommand{\runningauthor}[1]{\gdef\paper@runningauthor{#1}}
\makeatother

\titleformat{\section}[block]
  {\normalfont\bfseries\centering}
  {\thesection.}{0.5em}{}
\titlespacing*{\section}{0pt}{36pt plus 4pt minus 4pt}{9.5pt plus 2pt minus 2pt}

\titleformat{\subsection}[block]
  {\normalfont\bfseries}
  {\thesubsection.}{0.5em}{}
\titlespacing*{\subsection}{0pt}{31pt plus 4pt minus 4pt}{8pt plus 2pt minus 2pt}

\titleformat{\subsubsection}[block]
  {\normalfont\bfseries}
  {\thesubsubsection.}{0.5em}{}
\titlespacing*{\subsubsection}{0pt}{24pt plus 3pt minus 3pt}{6pt plus 2pt minus 2pt}

\titleformat{\paragraph}[runin]
  {\normalfont\itshape}
  {}{0pt}{}[.]
\titlespacing*{\paragraph}{\parindent}{12pt}{0.6em}

\makeatletter
\newlength{\tocsideinset}
\renewcommand{\@pnumwidth}{1.55em}
\renewcommand{\@tocrmarg}{2.55em}
\renewcommand{\@dotsep}{4.5}
\renewcommand*{\l@section}{\@dottedtocline{1}{0pt}{18pt}}
\renewcommand*{\l@subsection}{\@dottedtocline{2}{13pt}{23pt}}
\renewcommand*{\l@subsubsection}{\@dottedtocline{3}{26pt}{32pt}}
\newcommand{\makepapercontents}{%
  \par\vspace{0pt}%
  {\centering\bfseries Contents\par}%
  \vspace{5pt}%
  \begingroup
    \small
    \setlength{\parskip}{0pt}%
    \begin{adjustwidth}{\tocsideinset}{\tocsideinset}
      \@starttoc{toc}%
    \end{adjustwidth}
  \endgroup
}
\makeatother

\makeatletter
\AtBeginDocument{%
}
\makeatother

\makeatletter
\newif\ifshowtitleinfo
\showtitleinfotrue

\newcommand{\paper@journalline}{}
\newcommand{\paper@doiline}{}
\newcommand{\paper@rightsline}{}
\newcommand{\paper@authors}{}
\newcommand{\journalLine}[1]{\gdef\paper@journalline{#1}}
\newcommand{\doiLine}[1]{\gdef\paper@doiline{#1}}
\newcommand{\rightsLine}[1]{\gdef\paper@rightsline{#1}}
\newcommand{\paperauthors}[1]{\gdef\paper@authors{#1}}

\providecommand{\DefaultArticleTitleWidth}{\textwidth}
\newlength{\ArticleTitleWidth}
\newcommand{\setArticleTitleWidth}[1]{\setlength{\ArticleTitleWidth}{#1}}

\newcommand{\paper@printtitle}{%
  {\Large\selectfont
    \@tempdima=\textwidth
    \advance\@tempdima by -\ArticleTitleWidth
    \divide\@tempdima by 2
    \leftskip=\@tempdima plus 1fil
    \rightskip=\@tempdima plus 1fil
    \parfillskip=0pt
    \parindent=0pt
    \@title\par}%
}
\newcommand{\authorblock}[2]{%
  \begin{minipage}[t]{0.5\textwidth}
    \centering
    {\normalfont\scshape #1\par}%
    \vspace{2.4mm}%
    {\fontsize{7pt}{9pt}\selectfont #2\par}%
  \end{minipage}%
}
\renewcommand{\maketitle}{%
  \begingroup
  \thispagestyle{empty}%
  \setlength{\parindent}{0pt}%
  \vspace*{-13mm}%
  \ifshowtitleinfo
    {\fontsize{8pt}{10pt}\selectfont
      \paper@journalline\par
      \paper@doiline\par
      \paper@rightsline\par}%
  \else
    \vspace*{10.5mm}
  \fi
  \vspace*{16mm}%
  \paper@printtitle
  \vspace{6mm}%
  {\centering\large by\par}%
  \vspace{3.1mm}%
  {\centering\normalfont\paper@authors\par}%
  \vspace{12.4mm}%
  \endgroup
}
\makeatother

\makeatletter
\newtheoremstyle{mainclaim}
  {14pt plus 2pt minus 2pt}
  {14pt plus 2pt minus 2pt}
  {\itshape}
  {\parindent}
  {\normalfont}
  {}
  {0.5em}
  {\thmname{\textsc{#1}}\thmnumber{ \textup{#2.}}\thmnote{ \textup{(#3)}}}

\newtheoremstyle{plainnote}
  {14pt plus 2pt minus 2pt}
  {14pt plus 2pt minus 2pt}
  {\normalfont}
  {\parindent}
  {\normalfont}
  {}
  {0.5em}
  {\thmname{\textit{#1}}\thmnumber{\textup{ #2.}}\thmnote{ \textup{(#3)}}}
\makeatother

\newtheoremstyle{acknowledgement}
  {14pt plus 2pt minus 2pt}
  {14pt plus 2pt minus 2pt}
  {\normalfont\setlength{\parindent}{0pt}}
  {0pt}
  {\normalfont\bfseries}
  {}
  {0.5em}
  {\thmname{#1}.}
\makeatother

\theoremstyle{mainclaim}
\newtheorem{theorem}{Theorem}[section]
\crefname{theorem}{Theorem}{Theorems}
\Crefname{theorem}{Theorem}{Theorems}

\newcommand{\newaliastheorem}[3]{%
  \newaliascnt{#1}{theorem}%
  \newtheorem{#1}[#1]{#2}%
  \aliascntresetthe{#1}%
  \crefname{#1}{#2}{#3}%
  \Crefname{#1}{#2}{#3}%
}

\newaliastheorem{proposition}{Proposition}{Propositions}
\newaliastheorem{lemma}{Lemma}{Lemmas}
\newaliastheorem{corollary}{Corollary}{Corollaries}

\theoremstyle{plainnote}
\newaliastheorem{definition}{Definition}{Definitions}
\newaliastheorem{remark}{Remark}{Remarks}
\newaliastheorem{example}{Example}{Examples}
\newaliastheorem{convention}{Convention}{Conventions}

\newtheorem*{remark*}{Remark}
\newtheorem*{example*}{Example}

\theoremstyle{acknowledgement}
\newtheorem*{acknowledgements}{Acknowledgements}

\setlist[enumerate,1]{label=(\roman*),leftmargin=40pt,itemsep=0pt,topsep=2pt,parsep=0pt}
\setlist[enumerate,2]{label=(\alph*),leftmargin=24pt,itemsep=0pt,topsep=2pt,parsep=0pt}
\renewcommand{\.}{\hspace{0.05em}}
\usepackage{comment}